\documentclass[12pt]{amsart}
\usepackage{graphicx, url} 
\usepackage{amssymb,amsmath}

\newtheorem{thm}{Theorem}
\numberwithin{thm}{section}
\newtheorem{lem}[thm]{Lemma}
\newtheorem{remark}[thm]{Remark}
\newtheorem{cor}[thm]{Corollary}
\newtheorem{prop}[thm]{Proposition}

\def\cC{{\mathcal C}}
\def\cD{{\mathcal D}}

\def\cZ{{\mathcal Z}}
\def\PP{{\mathbb{P}}}
\def\FF{{\mathbb{F}}}
\def\ZZ{{\mathbb{Z}}}

\DeclareMathOperator{\cha}{char}
\DeclareMathOperator{\Aut}{Aut}
\DeclareMathOperator{\PGL}{PGL}
\DeclareMathOperator{\SL}{SL}
\title[Automorphisms of Chebyshev plane curves]{Automorphisms of Plane Curves defined from Chebyshev polynomials}

\author{Saeed Tafazolian and Jaap Top}
\date{}
\address{Departamento de Matem\'{a}tica - Instituto de Matem\'{a}tica, Estat\'{i}stica e Computação Cient\'{i}fica
(IMECC) - Universidade Estadual de Campinas (UNICAMP), Rua S\'{e}rgio Buarque de Holanda, 651, Cidade Universit\'{a}ria,  Zeferino Vaz, Campinas, SP 13083-859, Brazil}
\address{Bernoulli Institute for Mathematics, Computer Science, and Artificial Intelligence\\
Nijenborgh~9\\9747 AG Groningen\\ the Netherlands}
\email{saeed@unicamp.br}
\email{j.top@rug.nl}

\begin{document}

\begin{abstract}
For any $d\geq 4$ we determine the automorphism groups of the algebraic curves
\[
\mathcal{C}_d : y^d = \varphi_d(x),
\]
where $\varphi_d(x)$ denotes the Chebyshev polynomial of degree $d$, defined over a field $k$ with $p:=\operatorname{char}(k) \nmid 2d$. 
For $d>4$, it is $\mathbb{Z}/d\mathbb{Z}\times\mathbb{Z}/2\mathbb{Z}$,
except when $2d=p^r+1$ or $4d=p^r+1$ for some $r\geq 1$.
In these cases it is, respectively,
$(\mathbb{Z}/d\mathbb{Z}\times\mathbb{Z}/d\mathbb{Z})\rtimes S_3$
and $\mathbb{Z}/d\mathbb{Z}\times S_3$.
Two ingredients used here are a description of all total inflection points on $\cC_d$,
and knowledge of the group of fractional linear transformations that preserve the set of zeros of
$\varphi_d$. 
\end{abstract}

\maketitle
\section{Introduction and preliminaries}
For a field $k$ and an integer $d\geq 1$,
the $d$-th Chebyshev polynomial $\varphi_d(X)$
is the unique monic polynomial such that
in the ring of Laurent polynomials $k[t,t^{-1}]$
the equality 
\[ \varphi_d(t+t^{-1})=t^d+t^{-d}\]
holds. 
Equivalently, considering $x\mapsto \varphi_d(x)$ as a morphism $\PP^1\to\PP^1$ defined
over $k$, the polynomial is characterized by
the condition that the diagram\label{one}
\[    
\begin{array}{lcl}
{\PP}^1 & \stackrel{r\mapsto r^d}{\longrightarrow}& \mathbb{P}^1 \\
{\Big\downarrow}{\scriptstyle r\mapsto r+r^{-1}} && {\Big\downarrow}{\scriptstyle s\mapsto s+s^{-1}} \\
\mathbb{P}^1 & \stackrel{t\mapsto\varphi_d(t)}{\longrightarrow}& \mathbb{P}^1
\end{array}
    \]
commutes. From the diagram (see also \cite{TT1}
and \cite[Remark~6.2]{ABS}) it is readily deduced
that $\varphi_d(X)\in k[X]$ is separable, except
in the case that $\cha(k)$ divides $2d$.

Throughout this paper it is assumed that
$\cha(k)\nmid 2d$. Define the curve $\cC_d$ over $k$ as the closure in $\PP^2$ of the affine curve
given by
\[ y^d=\varphi_d(x).\]
Since $\varphi_d(X)\in k[X]$ is separable,
$\cC_d$ is a smooth curve. Our main goal is to
describe the group $\Aut(\cC_d)$ consisting of 
the automorphisms
of the curve $\cC_d$ defined over a fixed algebraic
closure of $k$. We restrict to
the case $d\geq 4$, or equivalently, the genus
$g(\cC_d)=(d-1)(d-2)/2\geq 2$, to ensure that
$\Aut(\cC_d)$ is a finite group. A
classical result \cite{Se} of B.~Segre
(see also \cite[Theorem~11.29]{HKT}) yields
that automorphisms of any smooth plane curve of degree at least 4 extend to  linear
automorphisms of $\PP^2$.

Two main tools in our study are the determination of all
total inflection points of $\cC_d$, for which a detailed proof is provided. Next, a description \cite{TTsym} of the group
\[H_d:=\left\{\mu\in\text{Aut}(\PP^1)\;:\; \mu(\{\text{the zeros of}\;\varphi_d\})=\{\text{the zeros of}\;\varphi_d\}\right\}\]
is used. The results obtained in
this way, are the following:
\begin{enumerate}
    \item[\rm(I)] The total inflection points of \( \mathcal{C}_d \) are precisely the \( d \) points in \( \mathcal{Z}(y) \cap \mathcal{C}_d \), except in the case where \( \operatorname{char}(k) = p \geq 3 \) and \( 2d = {p^m+1} \) for
    some $m\geq 1$. In this exceptional case there are \( 3d \) total inflection points. This result  extends \cite[Theorem~6.3]{ABS} (see Proposition~\ref{inflection}).

 \item[\rm(II)] If \( 2d = q+1 \) for some power \( q \) of \( p \), which is the
 situation described in (I), then \( \mathcal{C}_d \) is isomorphic to a Fermat curve of degree \( d \) and its automorphism group is \( (\mathbb{Z}/d\mathbb{Z} \times \mathbb{Z}/d\mathbb{Z})  \rtimes S_3 \) (see Propositions~\ref{ProFermat} and \ref{autFermat}); Although this fact follows from the general results of
\cite{HKT}, we provide here a short and elementary proof
which avoids the machinery used in that work.
    
    \item[\rm(III)]If \( d = 4 \), the (well known) result depends on $p$. To describe it, denote by $D_4$
    the dihedral group of order $8$. Also, given groups $G_1,G_2$ containing a unique involution $\iota_j$
    in their respective center, the associated central product (see see \cite[p.~29]{Gorenstein}) is
\[G_1\circ G_2=(G_1\times G_2/\langle (\iota_1,\iota_2)\rangle.\]

\begin{itemize}
    \item For \( p = 5 \), the automorphism group of \( \mathcal{C}_d \) is isomorphic to \( \mathbb{Z}/4\mathbb{Z}\circ\SL(2,3) \) (see Remark~\ref{re3.11}).
    \item For \( p = 7 \), the curve is isomorphic to the Fermat curve, and the automorphism group of \( \mathcal{C}_d \) is isomorphic to \( (\mathbb{Z}/d\mathbb{Z} \times \mathbb{Z}/d\mathbb{Z}) \rtimes S_3 \) (see Propositions~\ref{ProFermat} and \ref{autFermat}).
    \item For \( p \neq 2, 5, 7 \), the automorphism group of \( \mathcal{C}_d \) is isomorphic to \( \mathbb{Z}/4\mathbb{Z}\circ D_4 \) (see Theorem~\ref{d4}).
\end{itemize}
It is well known which groups occur as automorphism group of smooth plane quartic curves; see, e.g., \cite[Theorem~6.5.2 and Table~6.1]{Dolgachev}, \cite{Henn}, \cite[p.~62]{Vermeulen}, \cite{Bars}. Identifying $\cC_4\colon y^4=x^4-4x^2+2$
with the curve in their tables given by $X^4+Y^4+Z^4+aX^2Z^2=0$, where $a^2=8$, the result in {\rm(III)}
is immediate. We include our own explicit argument.

    \item[\rm(IV)] If \( 4d = q+1 \) for some power \( q \) of \( p \), then the automorphism group of \( \mathcal{C}_d \) is \( \mathbb{Z}/d\mathbb{Z} \times S_3 \) (see Theorem~\ref{thm3.15}).
    \item[\rm(V)] If $d>4$ and neither $2d-1$ nor $4d-1$ is a power of
$p$, then $\Aut(\cC_d)\cong\mathbb{Z}/d\mathbb{Z}\times\mathbb{Z}/2\mathbb{Z}$.
The same conclusion holds in characteristic zero (see Theorem~\ref{3.19}).
\end{enumerate}

The organization of this paper is as follows. In Section~\ref{S2}, we analyze the total inflection points of \( \mathcal{C}_d \) and clarify certain aspects of \cite{ABS}.  Section~\ref{S3} is dedicated to the classification of the automorphism groups for all $d\geq4$. 

\medskip

\section{Total inflection points}\label{S2}
Given a smooth curve $\cC\subset\PP^2$ of degree $d$, a total
inflection point of $\cC$ is a point $P\in\cC$
such that the associated tangent line to $\cC$  meets $\cC$ in $P$ with multiplicity $d$.

For $d\geq4$, automorphisms of any smooth plane curve $\cC$ are given by linear 
maps on $\PP^2$, hence they act as permutations on the set
of total inflection points of $\cC$.
We will
compute these points in the case of $\cC_d$,
assuming $d\geq 4$ and $\cha(k)\nmid 2d$. 
The computation extends the proof of
\cite[Theorem~6.3]{ABS}, where the special case
$p:=\cha(k)>2$ and $\overline{-1}\in \langle\, \overline{p}\,\rangle 
\subseteq(\mathbb{Z}/2d\mathbb{Z})^\times$
is considered. Moreover, our computations offer a refined analysis of some statements in \cite[p.~17]{ABS}. For instance, the reasoning there seems to suggest that $\mathcal{C}_d$ has either $d$ or $2d$ total inflection points, which does not align with the result in \cite[Theorem~6.3]{ABS}.

\begin{prop}\label{inflection}
Let $d\in\mathbb{Z}_{\geq 4}$. The only total inflection points on the curve $\cC_d\subset \PP^2$ defined as the closure of the affine equation $y^d=\varphi_d(x)$
over a field $k$ with $\cha(k)\nmid 2d$, are
the $d$ points in $\cZ(y)\cap\cC_d$, except in the
following situation (where there are $3d$ such points, namely the ones in the intersection
$\mathcal{Z}(y\cdot(x+2z)\cdot(x-2z))\cap \cC_d$):
\begin{enumerate}
    \item[] $\cha(k)=p\geq 3$ and $d=(p^m+1)/2$ for any $m\in\mathbb{Z}_{\geq 1}$ such that $d\geq 4$.
\end{enumerate}
\end{prop}
\begin{proof}
    The curve $\cC_d$ is defined by the
    homogeneous equation
    \[ y^d=z^d\cdot\varphi_d(x/z)=:\Phi_d(x,z).\]
    We consider all possibilities for a line
    $\ell\subset\PP^2$ with $\#(\ell\cap \cC_d)=1$.

    \noindent (1) The line at infinity ($z=0$)
    leads to $\ell\cap\cC_d$ given by $y^d=x^d$,
    hence $d$ distinct intersection points.\\
    (2) A line $\ell\colon x=az$ (for fixed $a$)
    results in an intersection given by $y^d=z^d\cdot \varphi_d(a)$. Hence this yields
    a single intersection point precisely when
    $a$ is a zero of $\varphi_d(X)$. The
    corresponding $d$ total hyperflexes are $(a:0:1)$,
    with tangent line $x=az$.\\
    (3) A line $\ell\colon y=az$ leads to the
    affine equation $\varphi_d(x)=a^d$.
    From the diagram on page~\pageref{one} it is
    clear that $t\mapsto \varphi_d(t)$ has,
    over points in $\PP^1\setminus\{\infty\}$,
    only ramification of index $2$. Hence such a line $\ell$ intersects $\cC_d$ in at least $d/2$ points.\\
    (4) What remains, is the case $\ell\colon y=ax+bz$ where $a\neq 0$. This gives
    $\#(\ell\cap \cC_d)=1$ precisely in
    the situation
    \[ (ax+bz)^d-\Phi_d(x,z)=
    (c_1x+c_2z)^d.\]
    Here an equality $a^d-1=c_1^d$ is obtained by
    considering the coefficient of $x^d$,
    while the coefficient of $z^d$
    results in $b^d-\varphi_d(0)=c_2^d$.

    We claim that $c_1\neq 0$. Indeed, the
    assumption $c_1=0$ implies that
    $\varphi_d(x)=(ax+b)^d-c_2^d$. Now observe
    that the
    two maps $t\mapsto\varphi_d(t)$ and
    $t\mapsto (at+b)^d-c_2^d$ are clearly distinct
    (the first has only $\infty$ as a ramification
    point with index $d$ while the other is 
    totally ramified over both $\infty$ and $-c_2^d$).
    As a consequence,
    \[  a^d\left(x+\frac{b}{a}\right)^d-\varphi_d(x)=(a^d-1)(x+c_2/c_1)^d.\]
    Writing
    \[\alpha:=a^d,\;\beta:=b/a,\;\gamma:=c_2/c_1,\;\text{and}\;x:=t+\frac1t,\]
    one obtains
    \[ \alpha\left(t+\beta+\frac1t\right)^d-t^d-t^{-d}=(\alpha-1)\left(t+\gamma+\frac1t\right)^d.\]
    Comparing coefficients results in the following
    system of equations.
    \label{system}\[ \arraycolsep=1.1pt
    \begin{array}{lrclr}
    t^{d-1}\colon& \alpha\beta d&=&(\alpha\!-\!1)\gamma d &(1)\\
    t^{d-2}\colon &\alpha\beta^2\binom{d}{2}+\alpha d&=&(\alpha\!-\!1)\left(\gamma^2\binom{d}{2}+d\right)&(2)\\
    t^{d-3}\colon&\alpha\beta^3\binom{d}{3}+2\alpha\beta\binom{d}{2}&=&(\alpha\!-\!1)\left(\gamma^3\binom{d}{3}+2\gamma\binom{d}{2}\right)&(3)\\
    t^{d-4}\colon&\alpha\!\left(\beta^4\binom{d}{4}+
    \beta^2d\binom{d-1}{2}+\binom{d}{2}\right)
    &=&
    (\alpha\!-\!1)\left(\gamma^4\binom{d}{4}+
    \gamma^2d\binom{d-1}{2}+\binom{d}{2}\right) &(4)
    \end{array}\]
    The remainder of the calculation distinguishes
    various possibilities for $d\in k$.
    \begin{itemize}
    \item Assume that $d=1$ in the field $k$. 
    The equality (2) results in the contradiction $\alpha=\alpha-1$.
    \item Next, consider the case $d=2$ in $k$,
    with $\cha(k)\neq 3$. Now equality (4) yields
    the contradiction $\alpha=\alpha-1$. 
    
        \item Assume $d-1\neq 0$ and $d-2\neq 0$ in $k$  (hence in particular $\cha(k)\neq 3$). Then (1) and (3) yield
        the system
        \[\arraycolsep=1.1pt
        \left\{\begin{array}{rcl}
        \alpha\beta&=&(\alpha-1)\gamma\\
        \alpha\beta^3(d-2)+6\alpha\beta&=&
        (\alpha-1)\gamma^3(d-2)+6(\alpha-1)\gamma,
        \end{array}\right.
        \]
        or equivalently (add a suitable multiple of
        the first equation to the second and
        multiply by $(d-2)^{-1}$), one obtains
        \[\arraycolsep=1.1pt
        \left\{\begin{array}{rcl}
        \alpha\beta&=&(\alpha-1)\gamma\\
        \alpha\beta^3&=&
        (\alpha-1)\gamma^3.
        \end{array}\right.
        \]
        This means $\alpha\beta^3=(\alpha-1)\gamma\cdot\gamma^2=\alpha\beta\gamma^2$, so since $\alpha\neq 0$ we are in one of the
        cases
        \[ \beta=0\;\;\text{or}\;\; 
        \beta=\gamma\;\;\text{or}\;\;\beta=-\gamma.\]
        If $\beta=0$ then (1) and (2) result in
        the contradiction $\alpha=\alpha-1$.\\
        If $\beta=\gamma$ (which has to be nonzero by what we just showed), then (1) yields the same contradiction $\alpha=\alpha-1$.

        The conclusion is that $\beta=-\gamma$, and this is nonzero. Hence (1) implies $\alpha=2^{-1}$. Substituting these values of $\alpha,\beta$ in (1)-(4) leads to
        the system
        \[\left\{
        \begin{array}{l}
        \gamma^2(d-1)=-2,\\
        \gamma^4(d-2)(d-3)+\gamma^2 \cdot12(d-2)=-12.
        \end{array}\right.\]
        So $\gamma^2=2/(1-d)$ and 
        \[ \frac{4}{(d-1)^2}(d-2)(d-3)-\frac{24}{d-1}(d-2)+12=0,\]
        resulting in $d\in\{3,2^{-1}\}\subset k$. 
        This is not possible if $\cha(k)=0$, since
        we assume $d\in\mathbb{Z}_{\geq 4}$.
        For $p=\cha(k)\geq 5$, two situations remain:\\
        (*) if $p \geq 7$ and  $d\equiv 3\bmod p$, one obtains
        \[  \left(t-i+\frac1t\right)^d+\left(t+i+\frac1t\right)^d=2t^d+2t^{-d}\]
        with $i$ any primitive $4$-the root of unity. Write $d=3+qn$ with $q$ a power of
        $p$ and $\gcd(n,p)=1$. Comparing
        coefficients of $t^{qn-q+2}$ results in $\pm 6n\equiv 0\bmod p$, a contradiction.
        Hence this case cannot occur.\\
        (*) If $2d\equiv 1\bmod p$, one obtains
        $\gamma^2=4$ and
        \[  \left(t-2+\frac1t\right)^d+\left(t+2+\frac1t\right)^d=2t^d+2t^{-d}.\]
        Taking $t=s^2$, this gives
        \[\left(s-\frac1s\right)^{2d}+\left(s+\frac1s\right)^{2d}=2s^{2d}+2s^{-2d}.\]
        Now write $2d=1+qn$
 with $q$ a power of $p$ and $\gcd(n,p)=1$. 
If $n=1$, the given equality indeed holds in $\FF_p[s,s^{-1}]$.
For $n>1$, considering the coefficient of $s^{2d-2q-2}$ leads to the contradiction $2n\equiv 0\bmod p$.
\item The final case to consider, is $\cha(k)=3$
and $d\equiv 2\bmod 3$. Writing $d=2+3m$,
this means $\binom{d}{3}\equiv m\bmod 3$ and
$\binom{d}{4}\equiv -m\bmod 3$.
The system of equations derived on p.~\pageref{system} then takes the form
\[ 
    \begin{array}{lrclcr}
    t^{d-1}\colon& 2\alpha\beta &=&2(\alpha\!-\!1)\gamma  &&(1)\\
    t^{d-2}\colon &\alpha\beta^2+2\alpha &=&(\alpha\!-\!1)\left(\gamma^2+2\right)&&(2)\\
    t^{d-3}\colon&m\alpha\beta^3+2\alpha\beta&=&(\alpha\!-\!1)\left(m\gamma^3+2\gamma\right)&&(3)\\
    t^{d-4}\colon&\alpha\!\left(-m\beta^4
    +1\right)
    &=&
    (\alpha\!-\!1)\left(-m\gamma^4+
    1\right) &&(4)
    \end{array}\]
      If $m=0$ in $k$, equation (4) gives the contradiction $\alpha=\alpha-1$. Thus $m\neq0$. Combining (1) and (3) results, as in the previous case, in the equalities $\beta=-\gamma$ and $\alpha=2^{-1}=2$. Then (2), (4)
    reduce to the system
    \[
    \left\{\begin{array}{c}2\gamma^2+1=\gamma^2+2,\\
    m\gamma^4+2=-m\gamma^4+1\end{array}\right.
    \]
    which is equivalent to $\gamma=\pm 2$ and $m=1\in k$. So in integers one obtains $d=2+3(1+3n)\equiv 5\bmod 9$, and in $k[t,t^{-1}]$ the equality
    \[  \left(t+2+\frac1t\right)^d+\left(t-2+\frac1t\right)^d=2t^d+2t^{-d}\]
    holds. Writing $t=s^2$ and $2d=1+q\ell$ with
    $q\geq 9$ a power of $3$ and $3\nmid \ell$,
    this takes the form
    \[ (s+s^{-1})(s^q+s^{-q})^\ell+(s-s^{-1})(s^q-s^{-q})^\ell=2s^{2d}+2s^{-2d}.\]
    In case $\ell=1$ (in $\mathbb{Z}$) this identity of Laurent polynomials indeed holds.
    In the remaining case $\ell\geq 2$,
    the coefficient of $s^{q\ell-2q-1}$ in the left-hand-side equals $2\ell\neq 0$ in $k$,
    hence the equality is not possible.
\end{itemize}
This completes the proof; the calculations show that indeed the total inflection points are the ones described in the statement of the proposition.
\end{proof}
\begin{remark}{\rm Note that in Proposition~\ref{inflection}, 
in any case where only $d$ total inflection points exist, these points are all
on a unique line (namely, on the line given by $y=0$). Similarly, in any case with
$3d$ total inflection points ($d\geq 4$), there
is a unique set of three lines such that
their union intersected with $\cC_d$ equals
this set of points (namely, the lines
defined by $y=0$ and $x=\pm 2z$).
}\end{remark}
\section{Automorphism groups}\label{S3}

 As a first step towards understanding the group
 of automorphisms of a plane curve $\cC_d$ 
 corresponding to $y^d=\varphi_d(x)$, we distinguish according to the number of
 total inflection points as given by
 Proposition~\ref{inflection}.

\begin{prop}\label{ProFermat}
Let \( \mathcal{C}_d \) be the curve defined by  
\[
\mathcal{C}_d: y^d = \varphi_d(x)
\]
over a field of characteristic \( p > 0 \).  
If \( 2d = q + 1 \) with \( q \) a power of \( p \), then \( \mathcal{C}_d \) is isomorphic over \( \mathbb{F}_{q^2} \) to the Fermat curve  
\[
\mathcal{F}: v^d = u^d + 1.
\]
\end{prop}

\begin{proof}
Consider the map $\mathcal{F}\to\PP^2$ given in affine coordinates by
\[
(u,v) \mapsto \left(x:=\frac{2u + 2}{u-1}\,,\, y:= \frac{av}{u-1}\right)
\]
where $a^d=2$ and so $a \in \FF_{q^2}$. 
Since $x,y$ generate the
function field $\FF_{q^2}(\mathcal{F})$,
this map defines a birational map to its
image. We
will show that its image is the curve $\mathcal{C}_d$, in other words we show
\[
2v^d= (u-1)^{d}y^d=(u-1)^{d} \varphi_d(\frac{2u+2}{u-1})
\]
in case $2d=q+1$. Equivalently, we claim that for $d=(q+1)/2$, in the polynomial
ring $\FF_p[u]$ the equality $(u-1)^{d} \varphi_d(\frac{2u+2}{u-1}) = 2u^d+2$ holds.
Indeed, write $u=s^2$. In $\mathbb{F}_p(s)\supset \FF_p(u)$ one observes
\[ \frac{2u+2}{u-1}=\frac{s+1}{s-1}+\frac{s-1}{s+1}.\]
From this, 
we obtain  
\[
(u-1)^{d} \varphi_d\left(\frac{2u+2}{u-1}\right)
=(s+1)^d(s-1)^d\left(\left(\frac{s+1}{s-1}\right)^d+\left(\frac{s-1}{s+1}\right)^d\right).
\]
Since \( 2d = 1 + q \), 
this equals
\[ (s^q+1)(s+1)+(s^q-1)(s-1)=2u^d+2.\]
Hence we described an explicit birational map from $\mathcal{F}$ to $\cC_d$.
Since both curves are smooth, the result follows.
\end{proof}

\begin{remark}\rm{
For \( 2d = q + 1 \), it is known that the curve \( y^d = \varphi_d(x) \) is maximal over \( \mathbb{F}_{q^2} \).  
Thus, by \cite[Theorem 1.1]{CKT}, the curve is isomorphic to the maximal Fermat curve of degree \( d=(q+1)/2 \).  
Here we provided a direct proof of this result.}

The explicit isomorphism presented above can
in fact be obtained using Proposition~\ref{inflection}. Namely, it is
well known that in terms of homogeneous
coordinates $u,v,w$ describing the Fermat curve $\mathcal{F}\colon v^d=u^d+w^d$,
the total inflection points of $\mathcal{F}$
(provided d-1 is not a power of the characteristic) are the $3d$ points in $\mathcal{F}\cap\mathcal{Z}(uvw)$.
Moreover, $\Aut(\mathcal{F})$ permutes
the three lines $\mathcal{Z}(u), \mathcal{Z}(v), \mathcal{Z}(w)$. The map used in
the proof of Proposition~\ref{ProFermat}
yields one of the possible ways to associate
these lines to the analogous lines for $\cC_d$ via a linear map.
\end{remark}

\begin{prop}\label{autFermat}
      Consider the curve $\mathcal{C}_d$ given by  \( y^d = \varphi_d(x) \) defined over a field of characteristic \( p > 0 \). 
If \( 2d = q + 1 \) with $q$ a power of $p$, then the automorphism group $G$ of the curve \( \mathcal{C}_d \) has order $6d^2$ and has a normal abelian subgroup $H \cong \mathbb{Z}/d\mathbb{Z} \times \mathbb{Z}/d\mathbb{Z}$ of order $d^2$ such that $G /H \cong S_3.$
\end{prop}

\begin{proof}From Proposition~\ref{ProFermat}, we know the curve $\mathcal{C}_d$
is isomorphic to the Fermat curve $\mathcal{F}$ of degree $d$.
Its $3d$ total inflection points (this number because of the linear isomorphism between
$\mathcal{C}_d$ and $\mathcal{F}$) are on the three coordinate lines.
These are the only lines containing $d$ such points: any other line
contains at most one on each coordinate line, and $d\geq4$.
Consequently every automorphism permutes the coordinate lines.
After rescaling the coordinates, write the equation as
$u^d+v^d+w^d=0$. The kernel of the action on these three lines consists
of the diagonal transformations and is isomorphic to
$\mathbb{Z}/d\mathbb{Z}\times\mathbb{Z}/d\mathbb{Z}$.
The coordinate permutations give a subgroup $S_3$, so
$G\cong(\mathbb{Z}/d\mathbb{Z}\times\mathbb{Z}/d\mathbb{Z})\rtimes S_3$.
\end{proof}
Both \cite{Leopoldt,Shioda} prove the same result on automorphisms of Fermat curves.
In the case $d=q+1$, which is disjoint from
the assumption in Proposition~\ref{autFermat},
the total inflection points on $\mathcal{F}$
are the intersection of $\mathcal{F}$ with
a set of more than $3$ lines, and the argument
presented above needs to be adapted accordingly.
Note that Leopoldt's result, although published in 1996, was already obtained in the early seventies. 

This completes the determination of
$\text{Aut}(\mathcal{C}_d)$ whenever this
curve has $3d$ total inflection points.
From now on we consider the remaining case.
 
 \begin{lem}\label{lemd}
     Suppose $\cC_d$ contains precisely $d$ total
     inflection points. Extending automorphisms of $\cC_d$ to linear
     automorphisms of $\mathbb{P}^2$, these
     restrict to automorphisms of the line $\mathcal{Z}(y)\cong \mathbb{P}^1$, and this defines
     a group homomorphism
     \[ \textrm{\rm Aut}(\cC_d)\longrightarrow \textrm{\rm Aut}(\mathbb{P}^1). \]
     Its kernel is central and consists of the automorphisms
     $(x,y)\mapsto (x,\zeta y)$ with $\zeta$
     ranging over the $d$-th roots of unity.
     Its image is the group $H_d$ of M\"{o}bius transformations $x\mapsto \frac{ax+b}{cx+\delta}$
     that permute the set of zeros of $\varphi_d(x)$.
 \end{lem}
\begin{proof}
    The set of total inflection points of $\cC_d$ is preserved under any automorphism (since
    automorphisms are given by linear maps.
    The assumption implies that $y=0$ defines the
    unique line containing all these inflection
    points, hence any automorphism restricts to
    a linear automorphism of this line, i.e.,
    to a M\"{o}bius transformation. Since 
    elements of $\text{\rm Aut}(\cC_d)$ permute
    the total inflection points, these
    M\"{o}bius transformations permute the zeros of $\varphi_d(x)$. Vice versa, if
    $x\mapsto \frac{ax+b}{cx+\delta}$ permutes the zeros of $\varphi_d(x)$, then
    \[  \varphi_d(\frac{ax+b}{cx+\delta})=e(cx+\delta)^{-d}\varphi_d(x)\]
    for some nonzero constant $e$. As a consequence,
    the given M\"{o}bius transformation is the image of the element
    \[ (x,y)\mapsto\left(\frac{ax+b}{cx+\delta}\,,\,\sqrt[d]{e}\frac{y}{cx+\delta}\right)\]
    in $\text{\rm Aut}(\cC_d)$.

    The tangent line at $(\alpha:0:1)$ is $x-\alpha z=0$.
These $d$ lines meet at the unique point $Q=(0:1:0)$.
An automorphism therefore preserves both $y=0$ and $Q$, is
given by a matrix 
\[
\begin{pmatrix}a&0&b\\0&e&0\\c&0&\delta\end{pmatrix},
\qquad e(a\delta-bc)\neq0
\]
as an element of $\text{PGL}(3,k)$. If it restricts to the identity map on the line $y=0$, then this matrix 
can be rescaled to
$a=\delta=1$ and $b=c=0$. The equation of the curve then implies
$e^d=1$. Conversely every such $[(x,y)\mapsto (x,ey)]\in \Aut(\cC_d)$ is in the kernel of the given 
homomorphism to $\Aut(\PP^1)$, and
commutes with every matrix of the displayed form.
\end{proof}

\begin{remark}
    \rm{
    It may be interesting to observe that Lemma~\ref{lemd} and the above proof generalize to
    any smooth plane curve $\cD\colon y^d=g(x)$ where $g(x)$ is a separable polynomial of degree $d\geq 4$.
    Namely, the $d$ zeros $\alpha$ of $g(x)$ yield total inflection points $(\alpha,0)\in\cD$.
    Hence if $\cD$ contains precisely $d$ total inflection points, then the (automatically linear)
    automorphisms of $\cD$ restrict to $\cZ(y)\cong\PP^1$ and one proceeds as in the proof of Lemma~\ref{lemd}.
    }
\end{remark}
Lemma~\ref{lemd} identifies $\Aut(\cC_d)/N$ with $H_d$
whenever $\cC_d$ has precisely $d$ total inflection points.

For completeness we  describe the automorphism group 
of $\cC_d$ in the specific case $d=4$.

\begin{thm}\label{d4}
Let $k$ be an algebraically closed field of characteristic
$\neq2,\neq5,\neq7$. The automorphism group $G$ of the curve
$y^4=\varphi_4(x)$ over $k$ has order $16$ and contains a central
subgroup $N\cong\ZZ/4\ZZ$ such that
\[
 G/N\cong\ZZ/2\ZZ\times\ZZ/2\ZZ.
\]
More precisely, $G\cong \mathbb{Z}/4\mathbb{Z}\circ D_4$.
\end{thm}

\begin{proof}
Since for $d=4$ we have $2d=p+1$ for $p=7$,
Proposition~\ref{inflection} and our 
assumption that $\cha(k)\neq 7$ imply that 
$\cC_4$ has precisely $4$ total inflection points. Hence Lemma~\ref{lemd} gives a central exact sequence
\[
 1\longrightarrow N\longrightarrow G\longrightarrow H_4
 \longrightarrow1,\qquad N\cong\ZZ/4\ZZ.
\]
With $\lambda^4=2$, set
$x=\lambda u$ and $y=\lambda v$. The curve becomes
\[
 v^4=u^4+a u^2+1,\qquad a=-2\lambda^2,\qquad a^2=8.
\]
The usual four-point cross-ratio classification, as used in
\cite[Section~6.5.2]{Dolgachev}, gives a Klein four-group for the
projective stabilizer of the zeros of $u^4+a u^2+1$, except when
$a(a^2-36)=0$ or $a^2+12=0$.
Here $a\neq0$, and $a^2=8$ shows that these exceptions occur only
in characteristics $7$ and $5$, respectively. Hence
$G/N\cong\ZZ/2\ZZ\times\ZZ/2\ZZ$ and $|G|=16$.

It remains to identify the extension.
Choose $i$ with $i^2=-1$. The automorphisms
\[
 z:(u,v)\mapsto(u,iv),\qquad
 s:(u,v)\mapsto(-u,v),\qquad
 t:(u,v)\mapsto(1/u,v/u)
\]
satisfy
\[
 z^4=s^2=t^2=1,\qquad z\ \text{is central},\qquad st=z^2ts.
\]
The images of $s$ and $t$ generate $G/N$, so these three elements
generate $G$. Put $r=st$. Then $r^2=z^2$ and $srs=r^{-1}$;
consequently $\langle r,s\rangle\cong D_4$ and
$\langle r,s\rangle\cap\langle z\rangle=\langle z^2\rangle$.
This proves $G\cong \mathbb{Z}/4\mathbb{Z}\circ D_4$. In particular, this central extension
does not split: every lift of the image of $s$ has nontrivial
commutator $z^2$ with every lift of the image of $t$.
\end{proof}

\begin{remark}\label{re3.11} {\rm
  In the case $\textrm{char}(k)=7$, the
  curve $\cC_4\colon y^4=\varphi_4(x)$
  contains $12$ total inflection points
  (see Proposition~\ref{inflection}) and in this
  case Proposition~\ref{autFermat} shows that
  the automorphism group is
  $\left(\mathbb{Z}/4\mathbb{Z}\times \mathbb{Z}/4\mathbb{Z}\right)\rtimes S_3$, of cardinality $96$.

 In the remaining case $\cha(k)=5$, the curve $\cC_4$ has precisely
$4$ total inflection points. With notations as in the proof of
Theorem~\ref{d4}, the same four-point classification gives $G/N=A_4$:
here $a^2=-12$, whereas $a(a^2-36)\neq0$.
Thus $|G|=48$.

Let $E$ be the inverse image in $G$ of the normal Klein four-group
of $A_4$. The same automorphisms $z,s,t$ used in the proof of
Theorem~\ref{d4} generate $E$, with the same relations.
The elements $u_0=zs$ and $v_0=zt$ satisfy
\[
 u_0^2=v_0^2=z^2,\qquad v_0u_0=z^2u_0v_0,
\]
and hence generate a quaternion group $Q$ of order $8$.
The presentation of $E$ gives
\[
 Q=\{1,z^2\}\ \cup\ \{g\in E\setminus N:g^2=z^2\}.
\]
Since $E$ is normal in $G$ and $N$ is central, this description
shows that $Q$ is normal in $G$.

Choose $g\in G$ whose image in $A_4$ has order $3$, and put
$w=g^4$. Since $g^3\in N$ and $N$ has order $4$, the element $w$
has order $3$ and has the same image as $g$. Its conjugation action
on $Q/\langle z^2\rangle$ cycles the three nonidentity elements.
Consequently
\[
 \langle Q,w\rangle\cong Q_8\rtimes C_3
 \cong\mathrm{SL}_2(\FF_3).
\]
Moreover,
\[
 G=N\cdot\langle Q,w\rangle,
 \qquad N\cap\langle Q,w\rangle=\langle z^2\rangle.
\]
It follows that
\[
 \Aut(\cC_4)\cong \mathbb{Z}/4\mathbb{Z}\circ\SL(2,3).
\]

Explicitly, $\cC_4$ in characteristic $5$ admits the automorphism
of order $3$ given by
\[
 (x,y)\mapsto
 \left(\frac{-2\beta x+(1-2\beta^2)}{(2\beta^2+2)x+\beta},
 \frac{\eta y}{(2\beta^2+2)x+\beta}\right),
 \qquad \eta=3(\beta^3+2\beta),
\]
where $\beta^4-4\beta^2+2=0$. 
}\end{remark}

\begin{prop}\label{main4d}
      Consider the curve \( \cC_d\colon y^d = \varphi_d(x) \) over a field $k$ of characteristic \(p\), with $4d=p^r+1$ for some $r>0$. 
Then $\cC_d$ admits the automorphism of order \( 3 \) given by
\[
(x,y) \mapsto \left(\frac{2x + 12}{2 - x}\,,\,  \frac{-4y}{2 - x}\right).
\]
\end{prop}

\begin{proof}
Write $q=p^r$, then $4d=q+1$
and therefore $\sqrt{-1}\not\in\FF_q$.
This implies
\[ (-4)^d=(1+\sqrt{-1})^{q+1}=(1-\sqrt{-1})(1+\sqrt{-1})=2.\]
From  \cite[Lemma~4.6]{TTsym} one obtains
\[
\varphi_d(x) = \frac{(2-x)^d}{2} \varphi_d\left(\frac{2x+12}{2-x}\right) \in k[x].
\]
Hence the given map indeed defines an automorphism of $\cC_d$.  A
straightforward computation shows that it 
has order $3$. 
\end{proof}

\begin{thm}\label{thm3.15}
Suppose that the characteristic of the field $k$ is $p\geq3$.
If $4d=q+1$, where $q$ is a power of $p$ and $d\geq4$, then
the automorphism group $G$ of the curve
\[
 \cC_d\colon y^d=\varphi_d(x)
\]
has order $6d$ and is isomorphic to $\ZZ/d\ZZ\times S_3$.
\end{thm}
\begin{proof}
The equality $4d=q+1$ rules out $d=4$, so $d>4$.
Also $2d-1$ cannot be a power of $p$, since otherwise the equality
$4d-1=2(2d-1)+1$ gives $0=1$ modulo $p$.
Proposition~\ref{inflection} and Lemma~\ref{lemd} therefore give a
central exact sequence
\[
 1\longrightarrow N\longrightarrow G\longrightarrow H_d
 \longrightarrow1,
 \qquad N=\{(x,y)\mapsto(x,\zeta y):\zeta^d=1\}.
\]
The determination of $H_d$ in \cite[Proposition~5.3]{TTsym} gives $H_d\cong S_3$.

Put $\epsilon=1$ if $d$ is even and $\epsilon=-1$ if $d$ is odd.
The automorphisms represented in homogeneous coordinates by
\[
 S=\begin{pmatrix}-1&0&0\\0&\epsilon&0\\0&0&1\end{pmatrix},
 \qquad
 T=\begin{pmatrix}2&0&12\\0&-4&0\\-1&0&2\end{pmatrix}
\]
are respectively a lift of $x\mapsto-x$ and the automorphism in
Proposition~\ref{main4d}. Direct multiplication gives
\[
 S^2=I,\qquad T^3=-64I,\qquad STS=-\tfrac14T^2.
\]
The classes of $S$ and $T$ in $\PGL(3)$ thus generate a subgroup
isomorphic to $S_3$ whose image is all of $H_d$. The subgroup of $\PGL(3)$ generated by the classes of $S,T$
meets $N$ trivially. Since $N$ is central, the resulting product
is direct, i.e., $G\cong\ZZ/d\ZZ\times S_3$.
\end{proof}

The remaining case also uses the determination of the group $H_d$ in \cite{TTsym}.

\begin{thm}\label{3.19}
Assume $d>4$ and $p\nmid 2d$. Suppose either that $p=0$, or that
$p>0$ and neither $2d-1$ nor $4d-1$ is a power of $p$. Then
$ \Aut(\cC_d)\cong\ZZ/2\ZZ\times\ZZ/d\ZZ$.
\end{thm}
\begin{proof}
By \cite[Theorem~1.1 and Remark~5.1]{TTsym}, the group of projective transformations permuting
the zeros of $\varphi_d$, under the given assumption on $d$, is
 $H_d=\{x\mapsto x,\ x\mapsto-x\}$.
Proposition~\ref{inflection} and Lemma~\ref{lemd} give the central
exact sequence
\[
 1\longrightarrow N\longrightarrow\Aut(\cC_d)
 \longrightarrow H_d\longrightarrow1,
 \qquad N\cong\ZZ/d\ZZ.
\]
Since $\varphi_d(-x)=(-1)^d\varphi_d(x)$, the map
\[
 \widetilde\sigma(x,y)=
 \begin{cases}(-x,y),&d\ \text{even},\\(-x,-y),&d\ \text{odd},\end{cases}
\]
is an involution lifting the nonidentity element of $H_d$.
It commutes with $N$, so
$\Aut(\cC_d)=N\times\langle\widetilde\sigma\rangle$.
\end{proof}

\begin{cor}\label{completeaut}
Let $d\geq4$ and $p:=\cha(k)\ge0$ and assume $p\nmid2d$.
Then the geometric automorphism group of $\cC_d$ is as follows:
\[
\Aut(\cC_d)\cong
\begin{cases}
 (\ZZ/d\ZZ\times\ZZ/d\ZZ)\rtimes S_3,
     & 2d=p^r+1\text{ for some }r\geq1,\\[2pt]
 \ZZ/4\ZZ\circ\SL(2,3),&d=4,\ p=5,\\[2pt]
 \ZZ/4\ZZ\circ D_4,&d=4,\ p\neq5,7,\\[2pt]
 \ZZ/d\ZZ\times S_3,&d>4,\ 4d=p^r+1\text{ for some }r\geq1,\\[2pt]
 \ZZ/d\ZZ\times\ZZ/2\ZZ,&\text{otherwise}.
\end{cases}
\]
\end{cor}
\begin{proof}
Combine Props.~\ref{ProFermat} and~\ref{autFermat},
Thms.~\ref{d4} and ~\ref{thm3.15} and~\ref{3.19}, and Remark~\ref{re3.11}.
\end{proof}

\normalsize{\textbf{Acknowledgment.} 
 \rm{The first author was partially supported by FAPESP grant No. 2023/08271-5, FAEPEX grant No.  2111/26  and CNPq grant No. 302774/2025–
4.}

\end{document}